\documentclass[11pt]{amsart}
\usepackage{amsmath,amssymb,amsbsy,amsfonts,amsthm,latexsym,amsopn,amstext,amsxtra,epic, euscript,amscd,indentfirst,verbatim,graphicx, hyperref,xcolor, setspace,tikz,tikz-cd}
\usepackage{tabularx} 
\usepackage{relsize} 
\usepackage[margin=1.in]{geometry}
\usepackage{enumitem} 
\usepackage{hyperref} 
\DeclareMathOperator{\Tr}{Tr}
\DeclareMathOperator{\Ric}{Ric}

\DeclareMathOperator{\Hess}{Hess}

\begin{document}

\numberwithin{equation}{section}
\newtheorem{theorem}{Theorem}[section]
\newtheorem*{theorem*}{Theorem}
\newtheorem{conjecture}[theorem]{Conjecture}
\newtheorem*{conjecture*}{Conjecture}
\newtheorem{proposition}[theorem]{Proposition}
\newtheorem*{proposition*}{Proposition}
\newtheorem*{``proposition"*}{``Proposition"}
\newtheorem{question}{Question}
\newtheorem{lemma}[theorem]{Lemma}
\newtheorem*{lemma*}{Lemma}
\newtheorem{cor}[theorem]{Corollary}
\newtheorem*{obs*}{Observation}
\newtheorem{obs}{Observation}
\newtheorem{example}[theorem]{Example}
\newtheorem{condition}{Condition}
\newtheorem{definition}[theorem]{Definition}
\newtheorem*{definition*}{Definition}
\newtheorem{proc}[theorem]{Procedure}
\newtheorem{problem}{Problem}
\newtheorem{remark}{Remark}
\newcommand{\comments}[1]{} 
\def\Z{\mathbb Z}
\def\Za{\mathbb Z^\ast}
\def\Fq{{\mathbb F}_q}
\def\R{\mathbb R}
\def\N{\mathbb N}
\def\C{\mathbb C}
\def\k{\kappa}
\def\grad{\nabla}
\def\M{\mathcal{M}}
\def\S{\mathcal{S}}
\def\pt{\partial}

\newcommand{\todo}[1]{\textbf{\textcolor{red}{[To Do: #1]}}}
\newcommand{\note}[1]{\textbf{\textcolor{blue}{#1}} \\ \\}

\title[Fundamental gaps on surfaces]{Log-Concavity and Fundamental Gaps on Surfaces of Positive Curvature}
 \author[Iowa State University]{Gabriel Khan} 
 \address[Gabriel Khan]{Department of Mathematics, Iowa State University, Ames, IA, USA.}
  \email{gkhan@iastate.edu}
 \thanks{G. Khan was supported in part by Simons Collaboration Grant 849022}

 \author[]{Xuan Hien Nguyen} 
  \address[Xuan Hien Nguyen]{Department of Mathematics, Iowa State University, Ames, IA, USA.}
  \email{xhnguyen@iastate.edu}
 \thanks{X.H Nguyen was supported in part by Simons Collaboration Grant 579756}

   \author[UCSB]{Malik Tuerkoen}
   \address[Malik Tuerkoen]{Department of Mathematics, University of California,  Santa Barbara, CA, USA.}
  \email{mmtuerkoen@ucsb.edu}
  
 \author[]{Guofang Wei}
\address[Guofang Wei]{Department of Mathematics, University of California,  Santa Barbara, CA, USA.}
\email{wei@math.ucsb.edu}
\thanks{M. Tuerkoen and G. Wei are partially supported by NSF DMS 2104704 and 2403557. }

\subjclass[2020]{53CXX, 58J50}
\date{}

\maketitle 

\begin{abstract}
We study the log-concavity of the first Dirichlet eigenfunction of the Laplacian for convex domains.  For positively curved surfaces satisfying a condition involving the curvature and its second derivatives, we show that the first eigenfunction is strongly log-concave. Previously, for general convex domains, the log-concavity of the first eigenfunctions were only known when lying in $\mathbb{R}^n$ and $\mathbb{S}^n$. 
 Using this estimate, we establish lower bounds on the fundamental gap of such regions. Furthermore, we study the behavior of these estimates under Ricci flow and other deformations of the metric.

 \bigskip
\noindent \textbf{Keywords.}  fundamental gap, Log-concavity, eigenfunctions\\
\end{abstract}
\section{Introduction}

Given a smooth domain $\Omega \subset M^n$ of a Riemannian manifold, we consider the Dirichlet eigenvalue problem
\begin{eqnarray*}
    -\Delta u = \lambda u \quad\textup{in }\Omega, \quad \quad u = 0 \quad\textup{on }\partial \Omega,
\end{eqnarray*}
where $\Delta$ is the Laplace-Beltrami operator. For domains which are bounded and connected, the eigenvalues
satisfy
$   0 < \lambda_1 <  \lambda_2  \le \lambda_3  \cdots \to \infty.$ As the first eigenvalue is simple, the fundamental (or mass) gap refers to the difference between the first two eigenvalues
\begin{equation*} 
\Gamma (\Omega) = 
	\lambda_2 - \lambda_1 >0
\end{equation*}
 of the Laplacian. This quantity has many applications in mathematics and physics, and can also be defined for a general Schr\"{o}dinger operator. 
 
 The log-concavity estimate of the first eigenfunction plays a key role in fundamental gap estimates for convex domains in $\mathbb R^n$ and $\mathbb S^n$ (see Section~\ref{history} for some history). 
 In this paper, we establish a very general result about the log-concavity of the first eigenfunction for convex domains
in an arbitrary Riemannian manifold following the approach of \cite{singer1985estimate, lee1987estimate} for $\mathbb R^n, \mathbb S^n$, see Theorem \ref{mainpropposition}. We then use this result on surfaces of positive curvature to obtain the following estimate for the first eigenfunction $u_1$.

\begin{theorem} \label{Main theorem}
Suppose $\Omega$ is a convex domain in a surface whose sectional curvature $\kappa$ is positive and satisfies
\begin{equation}\label{curvature condition}
    -\Delta \log \kappa  <  4 \lambda_1(\Omega) -5 \kappa -4C\left(\tfrac{C}{\kappa}+3\right)
\end{equation}
for some $C \in \mathbb R$. Then the first eigenfunction satisfies     
  \begin{equation}
  \label{eq:log-concavity}\Hess\,  (\log u_1) \le - C - \frac{\kappa}{2}. 
  \end{equation}
\end{theorem}

Using a result of Ling \cite{ling2006lower} to bound $\lambda_1$ from below, we can write \eqref{curvature condition} solely in terms of the curvatures, see Corollary \ref{Pinching of surfaces}.  Taking $C=0$, the assumption is always satisfied for $\mathbb S^2$ as $\lambda_1(\Omega) \ge 2$ for all convex domains $\Omega \subset \mathbb S^2$. In Subsection \ref{Ellipsoid examples}, we show all relatively round ellipsoids also satisfy the condition. 

Theorem \ref{Main theorem} has several consequences. First, the log-concavity estimate \eqref{eq:log-concavity} implies a lower bound on the fundamental gap.
    
\begin{cor} \label{Fundamental gap estimate on surfaces}
Given a convex domain $\Omega$ in a surface whose curvature satisfies \eqref{curvature condition} with  $C \in [- \kappa_{\min{}}, \infty)$, we have the following estimate on the fundamental gap:
\[\lambda_2 - \lambda_1> \frac{\pi^2}{D^2}+ \inf_\Omega \kappa + C.\]
\end{cor}

Secondly, Hamilton \cite{hamilton1988ricci} showed that compact surface with positive curvature converges to space form under normalized Ricci flow. Using an estimate of Hamilton for  $\log \kappa$ under Ricci flow, we give an explicit estimate on the time that   will satisfy \eqref{curvature condition} for a positively pinched surface, see Corollary \ref{Ricci flow and fundamental gaps}. Therefore the condition \eqref{curvature condition}, which is originally a fourth-order requirement on the metric can now be satisfied by a pinching of curvature if one is amenable to flowing the surface by Ricci flow for a fixed amount of time.

Lastly, equation \eqref{curvature condition} combined with Corollary \ref{Fundamental gap estimate on surfaces} shows that the inequality 
    \[
   \inf_{\Omega \text{ convex } \subset M} \Gamma(\Omega) D(\Omega)^2 \geq \pi^2
    \] is stable if the manifold $M$ is a $C^4$ perturbation of $\mathbb S^2$ (here, $D(\Omega)$ is the diameter of $\Omega$).
     For a fixed domain $\Omega \subset M$, the fundamental gap $\Gamma(\Omega)$ varies continuously under metric deformations, which is a consequence of the regularity of elliptic partial differential equations. However, the property above is required to hold for \emph{all} convex subsets when we perturb a metric, not simply on each individual convex set in a manifold. In contrast, when considering the Euclidean plane, any perturbation that introduces negative curvature gives a manifold violating the property above \cite{khan2024negative}.

 \subsection{An abridged history of the fundamental gap problem}  \label{history}

 The study of the fundamental gap has a long history and is important both in mathematics and physics. In 2011, Andrews and Clutterbuck \cite{andrews2011proof} proved the \emph{fundamental gap conjecture}: for convex domains $\Omega \subset \mathbb R^n$ with diameter $D$, 
 $$\Gamma(\Omega)  \ge 3\pi^2/D^2.$$  
The result is optimal, with the limiting case being rectangles that collapse to a segment. We refer to their paper for the history and earlier works on this important subject (see also the survey article \cite{daifundamental}). As discovered in \cite{singer1985estimate}, the log-concavity of the first eigenfunction is closely related to the fundamental gap estimate (see \eqref{Deltaw}). For convex domains in $\mathbb R^n$, the log-concavity of the first eigenfunction, $\Hess \, (\log u_1) \le 0$, was first obtained in \cite{Brascamp-Lieb1976}. The fundamental gap conjecture was proved by sharpening this estimate to obtain an optimal modulus of log-concavity \cite{andrews2011proof} (see also \cite{Ni}). In order to do so, Andrews-Clutterbuck applied a novel two-point maximum principle.

For convex domains in $\mathbb S^n$, Lee-Wang \cite{lee1987estimate} proved $\Hess \, (\log u_1) \le 0$ 
using the continuity method as in \cite{singer1985estimate}, which implies the gap estimate $ \ge \pi^2/D^2$. Later F.-Y. Wang \cite{wang2000estimation} improved the log-concavity estimate, see Section~\ref{sphere}. 
Recently, the last author, jointly with Dai, He, Seto, Wang (in various subsets) \cite{seto2019sharp, He-Wei2020, Dai-Seto-Wei2021},  generalized the two-point modulus of log-concavity estimate to convex domains in $\mathbb{S}^n$, and established the gap estimate  \[\lambda_2 - \lambda_1 > 3\pi^2/D^2,\] among other things. 

On the other hand, the size of the fundamental gap behaves very differently in spaces of negative or mixed curvature.
For general convex domains in $\mathbb H^n$, the first eigenfunction may not be log-concave, or even quasiconcave \cite{Shih1989}. This is reproved in 
\cite{bourni2022vanishing}, where it is furthermore shown that there is \emph{no lower bound} on the product of the fundamental gap and the square of the diameter for convex domains in hyperbolic space with arbitrary fixed diameter. In \cite{nguyen2022fundamental} it is also shown $(\lambda_2 - \lambda_1) D^2$ can be arbitrary small even for convex domains in $\mathbb H^n$ with log-concave first eigenfunction. Building off the results in hyperbolic space, the first and second named authors showed that for spaces of negatively pinched curvature, it is possible to find convex domains of arbitrary diameter whose fundamental gap is small \cite{khan2024negative}. Furthermore, in manifolds whose sectional curvature has a mixed sign, it is possible to find convex domains whose fundamental gap is arbitrarily small (although the diameter of these domains will need to be small). On the other hand, for rotationally symmetric domains the first eigenfunction will be log-concave for a large class of geometries (including hyperbolic space, see \cite{ishige2022power}).

Even though there is a vast amount of literature studying the convexity of solutions in $\mathbb R^n$ and $\mathbb S^n$, there are relatively few results when the curvature is non-constant and positive. In fact, it is still open whether the first eigenfunction of convex domains in $\mathbb {CP}^n$ is log-concave. 
On the other hand there are fundamental gap lower bound estimates for domains satisfying an interior rolling ball condition for general manifolds satisfying certain curvature bounds only (see \cite{oden1999spectral, oliverosewei2023integral}).

\subsection{Overview of paper}

The paper is structured as follows. In Section \ref{Barrier PDE section}, we show a general result about the log-concavity of the first eigenfunction for convex domains in an arbitrary Riemannian manifold, reducing the problem of establishing concavity estimates to finding a barrier function which satisfies a particular PDE inequality and does not grow too quickly at the boundary, see Theorem~\ref{mainpropposition}. Using this approach, we quickly reprove several earlier results on the fundamental gaps for convex domains of $\mathbb S^n$ \cite{lee1987estimate, wang2000estimation} by choosing appropriate barriers.

In Section \ref{Variable curvature section}, we find a particular barrier that works for surfaces of positive curvature. Doing so, we prove Theorem \ref{Main theorem}.  In subsection \ref{Ellipsoid examples}, we
 show all relatively round ellipsoids satisfy the necessary condition.

Finally, in Section \ref{Deformation of metrics section}, we discuss the relationship between fundamental gap estimates and metric deformations.  We also study the relationship between Ricci flow and fundamental gap estimates, and show that metrics which are sufficiently pinched satisfy Condition \eqref{curvature condition} after a short period of time under normalized Ricci flow (see Corollary \ref{Ricci flow and fundamental gaps}).

\textbf{Acknowledgements:} The last author would like to thank Gunhee Cho, Xianzhe Dai, and Shoo Seto for earlier discussions about the Lee-Wang paper \cite{lee1987estimate}. We would like to thanks the referees for careful reading and very helpful suggestions. 

\section{Log-concavity via the barrier PDE approach}
\label{Barrier PDE section}

In this section, we adapt the approach of \cite{singer1985estimate, lee1987estimate} to a general setting to show that establishing log-concavity can be reduced to constructing a barrier which satisfies a particular PDE inequality. Doing so recovers all the earlier work that could be obtained with one-point estimates, and will be useful for proving new results as well.

First we fix some notations. Let \[
R(X,Y)Z = \nabla_X \nabla_Y Z -  \nabla_Y \nabla_X Z - \nabla_{[X,Y]} Z \]
be the $(1,3)$ curvature tensor, and
$R_X$ the $(1,1)$-tensor given by \[
R_X (Y) = R(Y,X) X. 
\]

Given a strictly convex domain $\Omega \subset M^n$, assume 
there is a family of smooth strictly convex domains $\Omega(t) \subset \Omega$ such that $\Omega(0)$ is a sufficiently small ball in $\Omega$ and that $\Omega (1)=\Omega$.
This can be constructed explicitly when $M^n = \mathbb R^n$ using support functions. Within any symmetric space of positive curvature, mean curvature flow \cite{Huisken1986} will preserve the convexity of domains and on surfaces of positive curvature, curve shortening flow will do the same \cite[Prop.1.4]{Gage1990}. This process gives an almost round domain near time of singularity, one can then deform this small domain to a ball. In particular, one can construct such a deformation for any convex domains in $\mathbb S^n, \mathbb{ CP}^n$ or convex domains in a positively curved surface.

\begin{theorem}
    \label{mainpropposition}Let $\Omega \subset M^n$ be a compact strictly convex domain and $\Omega (t)$ a deformation as above. Let  $v (x, t)= \log u_1 (x, t)$, where $u_1$ is a positive first eigenfunction of the Laplacian on $\Omega(t)$ with Dirichlet boundary condition. 
Suppose there is a function $b:\Omega \to \mathbb{R}$ which
is uniformly bounded, 
solves
the PDE inequality 
\begin{multline}
   \Delta b >    2b^2-2\langle \nabla b, \nabla v\rangle + 2 \Tr [ (\Hess v + \nabla v \otimes \nabla v) \circ R_X] 
   +2b\Ric(X,X)  \\
     +(\nabla _{\nabla v}\Ric) (X,X)-2(\nabla _{X}\Ric) (X, \nabla v) \label{barrier PDE}
\end{multline}
 for all $t \in [0,1]$ whenever
 $X \in T_p\Omega$ a unit vector satisfying $\Hess\,v (X_p,X_p) + b(p)=0$ and $\Hess\,v (X_p,X_p) \ge \Hess\,v (Y_p,Y_p)$ for all unit vectors  $Y_p \in T_p\Omega$.
 
Then the function $v(x,1)$ satisfies the concavity estimate 
\begin{equation}   \label{concave v}
\Hess \, v + b\,  g \le  0 \quad \textup{on }\Omega.
\end{equation} 
\end{theorem}
Note that the derivatives of curvature terms only involve Ricci tensors, so this equation simplifies considerably for Einstein manifolds, not just symmetric spaces.

The proof of this follows along the same lines as \cite{singer1985estimate} for $\mathbb R^n$ and \cite{lee1987estimate} for $\mathbb S^n$. We provide a complete proof for a general manifold, filling some details in \cite{lee1987estimate}. One of the key computations in the proof is commuting derivatives. For convenience, we state these formulas now.

For a Riemannian manifold $M^n$, let $f\in C^4(M),$ and $\{e_i\}$ a basis of $T_pM$.  Denote
\begin{align*}
    f_{ij}=\textup{Hess}\,f(e_i,e_j), \ \ \ 
    f_{ij,k}=[\nabla_{e_k}\textup{Hess}\, f](e_i,e_j), \ \ \ 
     f_{ij,kl}=[\nabla_{e_l}\nabla_{e_k}\textup{Hess}\, f](e_i,e_j).
\end{align*}
We will use the following formulas to change order of derivatives, see e.g. \cite[Page 26, Lemma 2.1]{catino2020perspective}. 
\begin{align}
    f_{ij,k} &=f_{ik,j}+\sum_t f_tR_{tijk}, \label{3-commut} \\ 
   \sum_i f_{jj,ii} & = \sum_i\left[ f_{ii,jj}+2f_{ij}\Ric_{ij}+2f_i\Ric_{ij,j}-f_i\Ric_{jj,i} - \sum_s 2f_{si}R_{jsji} \right], \label{4-commut}
\end{align}
 where $R_{ijkl}= 
 \langle R(e_j,e_i)e_k,e_l\rangle$.

\begin{proof}  
First we note \eqref{concave v} is satisfied with a strict inequality near the boundary of $\Omega(t)$ for each $t \in [0,1]$ since  $\Omega(t)$ is strictly convex, and  $b$ is bounded (see, e.g., Lemma 3.4 in \cite{seto2019sharp}).  

Furthermore, when $\Omega(0)$ is a sufficiently small ball, then \eqref{concave v} also holds with a strict inequality on $\Omega(0)$. 
Indeed, for a small ball with radius $\rho$, if we scale the metric by $\rho^{-2}$ so it becomes a unit ball, then \eqref{concave v} becomes 
\begin{align*}
    \Hess\, v+ \rho^2 b \bar g \le 0, 
\end{align*}
where $\bar g = \rho^{-2}$ is the new metric, since $\Hess \, v$ is scale invariant. Since $\bar g$ can be made arbitrarily close to the Euclidean metric by making $\rho$ sufficiently small, and  the first eigenfunction of a unit ball in Euclidean space is strongly log-concave, the same result holds in our setting as well but in our case the radius of $\Omega(0)$ depends on $b$. 

 Now we argue by contradiction. Suppose that $\Hess \, v(X,X)_p+b(p)>0$ for some unit vector $X\in T_p\Omega,$ $p\in \Omega.$ Using the preceding argument and the continuity in $t$,
there must be a time $t_0 \in (0,1)$ such that the maximum of the mapping 
\begin{align*}
 U\Omega(t_0) \rightarrow \mathbb R, \quad X_q\mapsto  \Hess \,v(X_q,X_q) +b(q)
\end{align*}
is equal to zero, where $U\Omega(t)$ is the unit tangent bundle of $\Omega(t)$. That is, there exists a unit vector $X_p\in T_p\Omega(t_0)$ such that 
\begin{align*}
  0= \Hess \, v(X_p,X_p) +b(p)=\max_{Y_q \in U \Omega(t_0) } \left(\Hess \, v(Y_q,Y_q) +b(q) \right).
\end{align*}
By the estimates at the boundary for $\Omega(t)$, the maximum must be achieved at a point $p$ in the interior of $\Omega(t_0)$.

Let us denote $e_1=X_p$, 
and extend this vector to an orthonormal basis $\{e_j\}$ of $T_p\Omega(t_0)$. Now $e_1$ is the maximal direction of $\textup{Hess}\, v$ at $p$, which is symmetric, so $e_1$ is an eigenvector of $\textup{Hess}\, v$. Since $\{e_j\}$ is an orthonormal basis, we have at $p$, 
\begin{align*}
    \textup{Hess}\,v_p(e_1,e_j)= \langle \textup{Hess}\,v_p(e_1),\ e_j \rangle  =0\quad \textup{ for }j =2, \dots, n. 
\end{align*}
For any $q$ in a neighborhood of $p$, if we connect $q$ to $p$ by the unique minimal geodesic and parallel translate $e_1$ along the minimal geodesic to $q$, we obtain a smooth function $v_{11} = \Hess v(e_1, e_1)$ defined in a neighborhood of $p$, and the function $v_{11} + b$ achieves its maximum at $p$.  Hence by the maximum principle
\begin{align}
   e_i(v_{11})(p)+e_i(b)(p)&=0\quad \textup{for all }i =1,\dots, n  \label{gradient-zero} \\
   \Delta ( v_{11} ) (p)+\Delta b(p)&\leq 0. \nonumber
\end{align}

To compute $\Delta (v_{11}) (p)$, we extend the orthonormal basis $\{e_j\}$ of $T_p\Omega(t_0)$ to the geodesic frame in a neighborhood of $p$, i.e. $\{e_j\}$ is an orthonormal frame and $\nabla_{e_i} e_j (p)=0$ for all $i,j$. 
Then at $p$
\begin{align*}
   \Delta ( v_{11}) =\sum_ie_ie_iv_{11}=\sum_ie_i(e_iv_{11}).
\end{align*}
In a neighborhood of $p$, 
\begin{align*}
    e_i(v_{11}) &=v_{11,i}+2\sum_j\Gamma^j_{i1}v_{j1}.
\end{align*}
Therefore at $p,$ since $\Gamma_{ij}^k =0$, and $v_{1j}=0$ for $j \neq 1$, we have 
\begin{align*}
 e_ie_i(v_{11})&=v_{11,ii}+2\sum_je_i(\Gamma^j_{i1})v_{j1}+2\sum_j\Gamma^j_{i1}e_i(v_{j1})\\&
 =v_{11,ii}+2e_i(\Gamma^1_{i1})v_{11}. 
 \end{align*}
 Since $\{e_j\}$ is an orthonormal frame, 
$\Gamma^1_{i1} =0$ in the neighborhood of $p.$ In summary, we have
\begin{align*}
   \Delta ( v_{11}) (p)=\sum_i v_{11,ii} (p). 
\end{align*}
To compute the right hand side, 
recall that $v$ satisfies 
\begin{align}\label{vsPDE}
    \|\nabla v\|^2=-\lambda -\Delta v,
\end{align}
 so we commute the derivatives. Using \eqref{4-commut}, we have
 \begin{align*}
    \sum_iv_{11,ii}=\sum_iv_{ii,11}+2v_{11}\Ric_{11}-2\sum_{ij}v_{ji}R_{1j1i}+\sum_{i}2v_i\Ric_{i1,1}-\sum_iv_i\Ric_{11,i}.
\end{align*}
Taking the derivative of both sides of \eqref{vsPDE}, one has that 
\begin{align*}
    \sum_{i}v_{ii,11}=-2\sum_{i}(v_{i1}^2+v_{i1,1}v_i). 
\end{align*}
Using \eqref{3-commut} and \eqref{gradient-zero}, we have
\begin{align*}
   v_{i1,1} =v_{11,i}+\sum_jv_jR_{j1i1} 
   =-b_i+\sum_jv_jR_{j1i1}.
\end{align*}
Hence,    
\begin{align*} \sum_{i}v_{ii,11}&=-2\sum_{i}\left(v_{i1}^2-v_ib_i+\sum_j v_iv_jR_{j1i1}\right)\\ &=-2b^2+2\langle \nabla b, \nabla v\rangle -2\sum_{i,j} v_iv_jR_{j1i1}.
\end{align*}
Here in the second equation we used $v_{11} = -b, \ v_{i1} = 0$ for $i =2, \cdots, n$.   We therefore arrive at 
\begin{align*}
0&\geq\Delta (v_{11}+b)(p)\\ &= -2b^2+2\langle \nabla b, \nabla v\rangle -2\sum_{i,j} v_iv_jR_{j1i1}-2\sum_{i,j}v_{ij}R_{1j1i}-2b\Ric_{11}-\sum_{j}v_j\Ric_{11,j} \\ & \hspace{3.5in} +2\sum_{j}v_j\Ric_{1j,1}+\Delta b(p)\\
&=-2b^2+2\langle \nabla b, \nabla v\rangle -2\sum_{i,j} R_{j1i1}(v_iv_j+v_{ij})-2b\Ric_{11}-\nabla _{\nabla v}\Ric(e_1,e_1)
 \\ & \hspace{3.5in}  +2\nabla _{e_1}\Ric(e_1, \nabla v)+\Delta b(p)
\end{align*}
which contradicts \eqref{barrier PDE}.
\end{proof}

\subsection{Application to Round Spheres}  \label{sphere}

Although the inequality \eqref{barrier PDE} seems quite hard to check in general, when $b$ is a constant and the manifold has constant sectional curvature, it simplifies dramatically.

Let $M^n = \mathbb S^n$, and take $b$ to be constant. Then $\Ric \equiv n-1$, $R_X (Y) = Y - \langle X, Y \rangle X$. Therefore \eqref{barrier PDE} becomes
\[
0 > 2b^2 + 2 \sum_{i =2}^n (v_{ii} + v_i^2) + 2 (n-1)b = 2b^2 +2 (n-1)b + 2 \left( \Delta v + |\nabla v|^2 - v_{11} -v_1^2 \right).
\]
Using equation \eqref{vsPDE} and that at the maximal $v_{11} = -b$, this is equivalent to  
\begin{equation} \label{barrier on sphere}
0> 2b^2 +2 nb - 2\lambda_1 (\Omega (t)) -2 v_1^2.
\end{equation}
This is satisfied for all $t \in [0,1]$ when $b =0$, which recovers Lee-Wang estimate  \cite{lee1987estimate} $\Hess v \le 0.$ As $\Omega (t) \subset \Omega(1) =\Omega$ for all $t \in [0,1]$, by domain monotonicity $\lambda_1(\Omega) \le \lambda_1(\Omega(t))$. Hence \eqref{barrier on sphere} is satisfied for  all $t \in [0,1]$ when $b$ satisfies 
\[
0> 2b^2 +2 nb - 2\lambda_1 (\Omega).
\]
This is the case whenever $-\sqrt{ n^{2}+4 \lambda_1(\Omega)} -n  < 2b<   \sqrt{ n^{2}+4 \lambda_1(\Omega)} -n $. This recovers F.Y. Wang's estimate \cite{wang2000estimation}
\begin{equation} \label{Hess log u}
\Hess v \le  \frac{1}{2} \left( n-\sqrt{ n^{2}+4 \lambda_{1}(\Omega) } \right).
\end{equation}

It is worth noting that in equation \eqref{barrier on sphere}, there is an additional negative term $-2v_1^2$, which we have discarded. This term will be very important when we consider surfaces with variable curvature.

This log-concavity estimate on the first eigenfunction implies the following gap estimate. 
\begin{cor} \label{gap-sphere}
For $\Omega \subset \mathbb{S}^n$  a convex domain, its  fundamental gap has the following lower bound
\begin{equation}
    \label{eq:fg-lambda1}
 \lambda_2 - \lambda_1 > \frac{\pi^2}{D^2}  + \frac{1}{2} \left( \sqrt{n^2 + 4 \lambda_1} -1 \right), \end{equation}
where $D$ is the diameter of the domain $\Omega$. 
\end{cor}
In \cite[Theorem 1.2]{wang2000estimation} the weaker estimate, $\lambda_2 - \lambda_1 > \frac{\pi^2}{D^2}  +  (1- \frac{2}{\pi}) \left(\sqrt{n^2 + 4 \lambda_1} -1\right)$  was stated. 

Comparing the gap estimate $\lambda_2-\lambda_1> \frac{3\pi^2}{D^2}$  in \cite{Dai-Seto-Wei2021} to \eqref{eq:fg-lambda1}, one finds that \eqref{eq:fg-lambda1} is a stronger estimate when the first eigenvalue is large but the diameter is not too small. 
It shows that for convex domains of $\mathbb S^n$, the fundamental gap is large whenever the first eigenvalue is large, regardless of the diameter.
 In particular, for any convex thin strip, the gap goes to infinity whenever the inscribed radius $r$ goes to zero since $\lambda_1 >\frac{\pi^2}{4r^2} + \frac{n-1}{2}$ \cite{ling2006lower}. Now this is not true for thin convex domains on $\mathbb R^n$, as for thin rectangles the gaps stay bounded in term of the diameter. On the other hand, in hyperbolic space or in a space with mixed or negative curvature, the gap may go to zero as the inscribed radius $r$ goes to zero \cite{bourni2022vanishing, khan2024negative}. 

\begin{proof}[Proof of Corollary \ref{gap-sphere}]
Let $w= \frac{u_2}{u_1}$, where $u_1, \ u_2$ are the first and second eigenfunctions of the Laplacian
on $\Omega$ with Dirichlet boundary condition. Then $w$ can be extended to a smooth function on $\bar{\Omega}$ and satisfies 
\begin{equation}
    \Delta w + 2 \langle \nabla \log u_1, \nabla w \rangle = - (\lambda_2- \lambda_1) w.  \label{Deltaw}
\end{equation} with Neumann boundary condition on $\partial \Omega$ \cite{singer1985estimate}, see also \cite[(4.7)]{seto2019sharp}.  Consider $\mathbb S^n$ with weighted measure $u_1^2 \, \textrm{d} vol = e^{-(-2\log u_1)} \, \textrm{d}vol$. The Bakry-Emery Ricci curvature of this metric-measure space is $n-1 - 2 \Hess  \log u_1$, which is larger than $ 2\sqrt{n^2 + 4 \lambda_1} -1$ by \eqref{Hess log u}. Furthermore, the weighted Laplacian acting on $w$ is exactly the left hand side of \eqref{Deltaw}. By \cite{bakry2000some}, the first nonzero Neumann eigenvalue of a weighted Laplacian is $\ge \bar{\mu}_1 (n,D,a)$, where $a>0$ is a lower bound on the Bakry-Emery Ricci curvature, see \cite[Theorem 1,1]{andrews2012eigenvalue}. By \cite[Proposition 3.1]{andrews2012eigenvalue}, $\bar{\mu}_1 (n,D,a) \ge  \frac{\pi^2}{D^2} + \frac a2$.  Now the result follows as $a > 2\sqrt{n^2 + 4 \lambda_1} -1$. 
\end{proof}

\subsection{Barrier functions which do not satisfy the boundary conditions}

In Theorem \ref{mainpropposition}, we have assumed that $b$ is bounded, and all the barriers that we consider in this paper will be bounded. However, it is possible to weaken this assumption. In particular, whenever $b$ satisfies the 
boundary growth condition
\begin{equation} \label{Bounded growth at boundary}
    \lim_{z \to \partial \Omega} b(z) dist(z,\partial \Omega) =0,
\end{equation}
the concavity assumption 
$\Hess v + b g <0$ is satisfied near the boundary.

One point of interest is that there are barrier functions which satisfy inequality \eqref{barrier PDE} but do not satisfy the boundary requirements. For instance, on any rank-one symmetric space of compact type, the barrier $b = |\nabla v|^2$ satisfies inequality \eqref{barrier PDE}. However, this barrier corresponds to the principal eigenfunction being concave, which cannot hold (except in dimension $1$).

More interestingly, on a space whose curvature is positively pinched, the barrier
\begin{equation}
    b = \frac{\alpha}{n} |\nabla v|^2
\end{equation} satisfies Inequality \eqref{barrier PDE} when $\lambda_1$ is sufficiently large and $\alpha$ is between $0$ and $1$. The precise values of $\alpha$ and $\lambda_1$ which make this function a solution to \eqref{barrier PDE} depend on the curvature pinching and the derivative of the Ricci curvature in a fairly delicate way.\footnote{The derivation to show this solves \eqref{barrier PDE} incorporates ideas from the Li-Yau gradient estimate \cite{li1986parabolic}, but is a fairly lengthy calculation, which is why we have omitted it.}  However, it is possible to show that this barrier fails at every point in the boundary under Dirichlet conditions (see Lemma 3.4 of \cite{seto2019sharp}). 
 Therefore, these barriers cannot be used directly to study the fundamental gap problem. However, they can be used to establish log-concavity properties for eigenfunctions which satisfy other boundary conditions (e.g., non-homogeneous positive Dirichlet conditions).

\section{Surfaces with non-constant positive curvature}
\label{Variable curvature section}

In this section we apply Theorem~\ref{mainpropposition} to surfaces to prove Theorem~\ref{Main theorem}. The crux is to find an appropriate barrier function, which will not be constant.

Before we present the proof, note that by \cite[Prop. 1.4]{Gage1990}, a strictly convex domain $\Omega \subset M^2$ is still strictly convex when the boundary is deformed by curve shortening flow. Moreover when $M^2$ has positive curvature, the minimum of the geodesic curvature of the curve is nondecreasing (i.e. the $\alpha$ in the proof of Prop. 1.4 in \cite{Gage1990} can be chosen to be zero). Hence we have a family of strictly convex domain $\Omega(t)$ such that $\Omega (1) = \Omega$ and $\Omega(0)$ is very close to a small ball with $\Omega(t) \subset \Omega$ as needed for Theorem \ref{mainpropposition}. 

\begin{proof}[Proof of Theorem~\ref{Main theorem}] When $n=2$, $\Tr [ (\Hess v + \nabla v \otimes \nabla v) \circ R_X]$ simplifies considerably. As before, at the maximal $X_p$, denote $e_1 =X_p$ and $e_2$ its orthogonal complement so that $\{e_1, e_2\}$ form an orthonormal basis. Then 
 \begin{align*}
    \Tr [ (\Hess v + \nabla v \otimes \nabla v) \circ R_X] & = \kappa (v_2^2+v_{22}) \\ & = \kappa (-\lambda_1 -v_1^2-v_{11})  = \kappa (-\lambda_1 -v_1^2 +b). 
 \end{align*} Here we used \eqref{vsPDE} in the second equality.  Also $\Ric =\kappa\, g$, hence $(\nabla _{\nabla v}\Ric) (X,X)-2(\nabla _{X}\Ric) (X, \nabla v) = v_1\kappa_1 + v_2 \kappa_2 -2 v_1 \kappa_1 = v_2 \kappa_2 - v_1 \kappa_1.$ Therefore \eqref{barrier PDE} becomes\begin{align*}
    \Delta b & >  2b^2-2\langle \nabla b, \nabla v\rangle - 2 \kappa  (\lambda_1 + v_1^2) +4b\kappa
     +v_2 \kappa_2 - v_1 \kappa_1. 
 \end{align*}
 Note that if $b =0$, the problem terms are $v_2 \kappa_2 - v_1 \kappa_1$.  One can use $- 2 \kappa v_1^2$ to control $v_1 \kappa_1$, so one would like to choose $b$ suitably to cancel $v_2 \kappa_2$. 
   Now if we choose   
\begin{equation} \label{Barrier for variable curvature}
    b = \tfrac{\kappa}{2} + C, 
\end{equation}
where $C \in \mathbb R$ is a constant (we add the constant to allow more flexibility), this does the job. The expression above becomes
\begin{align*}
\Delta \tfrac{\kappa}{2} & >   2( \tfrac{\kappa}{2} + C)^2-2v_1 \kappa_1 -  2 \kappa  (\lambda_1 +v_1^2)  +4\kappa ( \tfrac{\kappa}{2} + C) \\
& =  2( \tfrac{\kappa}{2} + C)(\tfrac{5\kappa}{2} + C )  -  2 \kappa  \lambda_1 - 2\kappa (v_1 + \tfrac{\kappa_1}{2\kappa})^2 + \tfrac{\kappa_1^2}{2\kappa}.
 \end{align*}
 Multiplying the above by $\tfrac{2}{\kappa}$ gives
 \begin{align*}
 \tfrac{\Delta \kappa}{\kappa} - \tfrac{\kappa_1^2}{\kappa^2} &> -4 \lambda_1 + (1+ \tfrac{2C}{\kappa}) (5 \kappa + 2C) - 4 (v_1 + \tfrac{\kappa_1}{2\kappa})^2 \\
 &= -4 \lambda_1 + 5\kappa + 4C(3+\tfrac{C}{\kappa})- 4 (v_1 + \tfrac{\kappa_1}{2\kappa})^2.  
 \end{align*}
 This is satisfied when 
 \[ \Delta \log \kappa > -4 \lambda_1(\Omega(t)) + 5\kappa + 4C(3+\tfrac{C}{\kappa}).
 \]
 Again since $\Omega(t) \subset \Omega (1)$, the above is satisfied for all $t \in [0,1]$ if it is satisfied at $t=1$. Applying Theorem~\ref{mainpropposition} finishes the proof. 
 \end{proof}
 
It should be noted that there are two algebraic cancellations in the preceding argument. Because the curvature can be expressed in terms of a single function, it is possible to build a barrier which eliminates the problem $v_2 \kappa_2$ term.  It is remarkable that after completing the square to control the $v_1 \kappa_1$ term, the extra term is controlled when we rewrite the inequality so that the leading term is $\Delta \log \kappa$. 
 
With the log-concavity estimate in Theorem~\ref{Main theorem}, we obtain the gap estimate in Corollary~\ref{Fundamental gap estimate on surfaces} as in the proof of Corollary~\ref{gap-sphere}. Namely, the Bakry-Emery Ricci is $\ge \inf_\Omega \kappa -2 (-C -\tfrac{\kappa}{2} )\ge  2C + 2 \inf_\Omega \kappa$, and the fundamental gap $\lambda_2 -\lambda_1 \ge \frac{\pi^2}{D^2}+ \inf_\Omega \kappa + C$.

\subsection{Examples}
\label{Ellipsoid examples}

In this subsection, we specialize condition \eqref{curvature condition}  for $C= 0, -\frac{\kappa_{\min}}{2}$ and find examples of surfaces which satisfy these conditions for all convex domains. In order to do so, let us first prove a corollary which gives a way to verify this condition solely by computing the curvature and its derivatives.

\begin{cor} \label{Pinching of surfaces}
If the sectional curvature $\kappa$ of $M^2$ satisfies $\kappa \ge \kappa_{\min{}} >0$ and 
\begin{equation}
    -\Delta \log \kappa  < 6\kappa_{\min{}} - 5 \kappa \ \    \ \ \mbox{or}  \ \ \ 
    -\Delta \log \kappa  < 11\kappa_{\min{}} - 5 \kappa,   \label{k assumption}
\end{equation}
then for any convex domain $\Omega \subset M$, its first Dirichlet eigenfunction $u_1$ satisfies 
\begin{equation}
    \Hess\,  (\log u_1) \le - \frac{\kappa}{2} \ \ \mbox{or} \ \ \ \Hess\,  (\log u_1) \le0,  \label{logconcavity-ests}
    \end{equation} respectively. In particular,  when
 $M^2$ is positively pinched with $0<\kappa_{\min{}} \le \kappa \le \alpha \kappa_{\min{}}$, then assumptions \eqref{k assumption} are satisfied if  $1\le \alpha <\tfrac{6}{5}$, and $
- \Delta \log \kappa < (6-5\alpha) \kappa_{\min{}},$ or $1\le \alpha <\tfrac{11}{5}$, and $
- \Delta \log \kappa < (11-5\alpha) \kappa_{\min{}}$.
\end{cor}

\begin{proof}
Given any domain $\Omega$, Ling \cite{ling2006lower} showed that when the Ricci curvature satisfies $\Ric\ge  (n-1) \kappa_{\min{}}$ and the boundary $\partial \Omega$ has non-negative mean curvature, the first Dirichlet eigenvalue satisfies
\begin{equation*}
    \lambda_1 (\Omega) \geq \frac{1}{2}(n-1) \kappa_{\min{}}+\frac{\pi^{2}}{d^{2}},
\end{equation*}
where $d$ is the diameter of the largest inscribed ball. Furthermore, Myers's theorem implies that the diameter of $M$ satisfies $\textrm{diam}(M)^2 < \frac{\pi^2}{\kappa_{\min{}}}$. Since $\textrm{diam}(M) > d(\Omega)$, we have $\lambda_1 (\Omega) \geq \frac{1}{2}(n+1) \kappa_{\min{}}$. Combining this with \eqref{curvature condition} and letting  $C= 0, -\frac{\kappa_{\min}}{2}$ respectively, then \eqref{eq:log-concavity} gives \eqref{logconcavity-ests}. 
\end{proof}

With this result, we can show all relatively round  ellipsoids satisfy \eqref{k assumption} for all convex domains. Hence the first eigenfunctions are all log-concave.

\begin{example}
For any convex domain $\Omega$ in an ellipsoid of the form 
\[ \frac{x^2}{\epsilon^2}+y^2+ z^2=1,  \] the first Dirichlet eigenfunction $u_1$ satisfies $\Hess\,  (\log u_1) \le 0$  whenever $\epsilon \in (.8064, 1]$. 
\end{example}

We will take for granted the fact that the quantity $\Delta \log \kappa - 5\kappa$ is minimized at the point of maximal curvature for such an ellipsoid. This can be verified in Mathematica. 
For such ellipsoids, the MTW curvature 
has been computed explicitly in \cite[Section 6]{figalli2010ma}, in order to study the regularity of optimal transport on these surfaces. We use some of the computations from this paper. 
\begin{proof}
As in \cite{figalli2010ma}, we can view an ellipsoid as a surface of revolution by rotating  the graph $z = F_\epsilon (x)$  along the $x$-axis, where $F_\epsilon: [-\epsilon, \epsilon] \rightarrow \mathbb R_+$ is defined by
$$ F_\epsilon (x) = \left(1- \tfrac{x^2}{\epsilon^2} \right)^{1/2}.  $$
Its metric is $$ dr^2 + f^2(r) d \theta^2, $$
where $r(x) = \int_{-\epsilon}^x \sqrt{1 + F'(u)^2} du$, \ $f(r(x)) = F_\epsilon (x).$ Then the Gaussian curvature $\kappa = - \frac{f''}{f}$, and the Laplace operator  $\Delta = \frac{\partial^2}{\partial r^2} + \frac{f'}{f}  \frac{\partial}{\partial r} + \frac{1}{f^2} \frac{\partial^2}{\partial \theta^2}.$ Hence 
\begin{eqnarray*}
   \Delta \log \kappa - 5\kappa  &  = & \tfrac{\partial^2}{\partial r^2} \left[ \log (-f'') - \log f \right] + \tfrac{f'}{f} \tfrac{\partial}{\partial r} \left[ \log (-f'') - \log f \right] - \left(-5 \tfrac{f''}{f} \right) \\
    & = & \frac{f''''}{f''}  - \left( \frac{f'''}{f''} \right)^2 + \frac{f' f'''}{f f''} + 4 \frac{f''}{f}.
\end{eqnarray*}
At $(0,0,1)$, which achieves the maximal curvature,  $F_\epsilon (0) =1, \ F_\epsilon'(0) = F_\epsilon''' (0) =0, \  F_\epsilon'' (0) = -\tfrac{1}{\epsilon^2}, \ F_\epsilon'''' (0) = -\tfrac{3}{\epsilon^4}$, and $f=1,\  f'=f'''=0, \ f'' = -\tfrac{1}{\epsilon^2}, \ f'''' = F_\epsilon''''(0) - 4 (F_\epsilon''(0))^3.$ 

Hence at $(0,0,1)$, $\Delta \log \kappa - 5\kappa = - \epsilon^{-2} -4 \epsilon^{-4}.$ Now $- \epsilon^{-2} -4 \epsilon^{-4} + 11 > 0$ when $0.80634 \approx (\frac{1 + \sqrt{177}}{22})^{1/2} < \epsilon \le 1$. 
\end{proof}

If we drop the assumption that the ellipsoid is rotationally symmetric, it is possible to use Mathematica to minimize the quantity $\Delta \log \kappa -5\kappa + 12 \kappa_{\min{}} - \frac{\kappa_{\min{}}^2}{\kappa}$. Doing so, we find the following (See Definition \ref{Spectral dominance property} for a definition of spectral dominance).

\begin{example}
An ellipsoid of the form 
\[ \frac{x^2}{a}+\frac{y^2}{b}+ \frac{z^2}{c}=1  \]
with $a \leq b \leq c$ satisfies the spectral dominance property whenever 
\begin{equation}
    \frac{c}{a} \leq 1.132
\end{equation}
\end{example}

\section{Deformations of metrics}

\label{Deformation of metrics section}

\subsection{The spectral dominance and $\Gamma_\delta$ properties}

Throughout this section, it will be convenient to have terminology to keep track of the various gap estimates as well as the hypotheses required to satisfy them. To this end, we provide some definitions.

\begin{definition}
A space $M$ satisfies the fundamental gap $\delta$ (denoted $\Gamma_\delta$) property if for \emph{all} 
convex domains $\Omega$, the following estimate holds:
\begin{equation*}
     \lambda_{2}(\Omega)-\lambda_{1}(\Omega) \geq \frac{\delta \pi^{2}}{D^{2}(\Omega)},
\end{equation*}
where $D(\Omega)$ is the diameter of $\Omega$. 
\end{definition}

In other words, a surface satisfies $\Gamma_\delta$ if the following inequality holds.
\begin{equation*}
    \inf \{ \Gamma(\Omega) \, D^2(\Omega) ~|~ \Omega \subset M \textrm{ convex}   \} \geq \delta \pi^{2}.
\end{equation*}

In view of Theorem \ref{Main theorem}, it is natural to consider surfaces which satisfy the curvature condition \eqref{curvature condition} for \emph{all} convex domains. We take $C= -\frac{\kappa_{\min{}}}{2}$, so that the condition holds for the widest possible class of surfaces. However, it is possible to consider other values of $C$ as well. 

\begin{definition}\label{Spectral dominance property}
 A surface $M$ satisfies the spectral-dominance property if it has positive curvature and for all convex domains $\Omega \subset M$ and points in $M$, we have the inequality
\begin{equation*}
    \Delta \log \kappa  > \frac{4}{\kappa}  \left( \frac{\kappa}{2} -  \frac{\kappa_{\min{}}}{2} \right)^2  - 4 \lambda +8\left(\frac{\kappa}{2} -  \frac{\kappa_{\min{}}}{2}\right).
\end{equation*}
\end{definition}

In other words, a surface is spectrally dominant whenever
\[\inf \{ 4 \lambda(\Omega)  ~|~ \Omega \subset M \textrm{ convex}\} + \inf_{z \in M} \left(\Delta \log \kappa(z) - \frac{ (\kappa -\kappa_{\min{}})^2 }{\kappa} - 4 (\kappa -\kappa_{\min{}})  \right)  >0, \]
which is to say that the spectrum ``dominates" the curvature.

\subsection{General observations}

Corollaries \ref{Pinching of surfaces} and \ref{Fundamental gap estimate on surfaces} imply that the spectral dominance property holds for $C^4$-small deformations of $\mathbb S^2$, and thus that these spaces satisfy $\Gamma_1$.

Eigenvalues and eigenfunctions depend on the metric in a continuous way (in a smooth enough topology), so it may be tempting to conclude that since the round sphere satisfies $\Gamma_3$, a nearly-round sphere automatically satisfies $\Gamma_1$, simply by continuity. However, this continuity argument \emph{fails} for Euclidean space, which satisfies $\Gamma_3$ but can be deformed in $C^\infty$ fashion to a metric with some negative sectional curvature. Once there is any negative sectional curvature, $\Gamma_\delta$ fails for all $\delta>0$ as it is possible to build small convex domains with arbitrarily small fundamental gaps. So the fundamental gap $\delta$-property does \emph{not} depend on the metric in a continuous way.


The reason is that the eigenvalues of domains do not depend on the metric in a \emph{uniformly} Lipschitz way. In particular, for a given domain, the Lipschitz constant of the first two eigenvalues depends strongly on the geometry of the domain. For instance, if a domain is Gromov-Hausdorff close to a submanifold of positive codimension, even small perturbations can greatly affect the spectrum. This effect plays a central role in showing that fundamental gaps can be arbitrarily small when there is even a single tangent plane of negative sectional curvature.  

These properties explain why we cannot expect the $\Gamma_\delta$ condition to be deformation stable in general. However, the results of this paper raise the following questions. 

\begin{question}
Under what conditions is the $\Gamma_\delta$ condition deformation stable, and in what topology? For instance, do sufficiently $C^2$-small deformations of a round sphere satisfy $\Gamma_{2.99}$?
\end{question}

Although the $\delta$ in the $\Gamma_\delta$ estimate will not always depend continuously on the metric, in all the examples we have considered, it appears to be upper-semicontinuous in the metric.

\begin{question}
Let $(M_i,g_\infty)$ be a sequence of Riemannian manifolds which all satisfy $\Gamma_\delta$ and converge to an Alexandrov space\footnote{Note that $\Gamma_\delta$ implies nonnegative sectional curvature, so the Gromov-Hausdorff limit of such spaces is Alexandrov.} $M_\infty$. Does $(M_\infty,g_\infty)$ also satisfy $\Gamma_\delta$?
\end{question}

At first, Gromov-Hausdorff convergence might seem to be too rough for such a result to hold. However, there is another fourth-order curvature condition for which this is the case, which suggests that such a result might be possible \cite{villani2008stability}.

\subsection{Fundamental gaps and Ricci flows}

One question of interest is whether it is possible to weaken the assumption on the curvature in equation \eqref{curvature condition} to a condition which only depends on the curvature, and not its derivatives.

If one is willing to deform the metric by Ricci flow, it is possible to obtain such an estimate after some amount of time. In particular, Hamilton's Harnack estimate can be used to gain control of the term  $\Delta \log \kappa$, which implies the following corollary.

\begin{cor}\label{Ricci flow and fundamental gaps}
Let $(M^2,g_0)$ be a Riemannian surface whose sectional curvature positively $\frac{6}{7}$-pinched and whose average sectional curvature is $1$. Consider the solution to the normalized Ricci flow 
\[ \partial_t g = -  Ric(g) + g. \]
At time $t = \log(3/2)$, the manifold $(M^2,g_t)$ is spectrally dominant.
\end{cor}

In fact, for any given ratio $\frac{\kappa_{\min{}}}{\kappa_{\max{}}}$ close enough to $1$, there is a time $t$ at which the metric is spectrally dominant. As $\frac{\kappa_{\min{}}}{\kappa_{\max{}}}$ converges to $1$, the smallest such time for which we can establish this estimate converges to $\log(6/5) \approx .182$, but with more careful Ricci flow analysis and sharper eigenvalue bounds, one might be able to lower this time somewhat.

\begin{proof}

From Hamilton's Harnack estimate on surfaces (Theorem 6.3 of \cite{hamilton1988ricci}), under the normalized Ricci flow, we have the following estimate
\begin{equation*} 
 \Delta \log \kappa + \kappa - 1 \geq-  \frac{e^{ t} }{e^{ t}-1   }.
\end{equation*}
Corollary \ref{Pinching of surfaces} therefore holds as soon as 
\begin{equation*}
  1 -  \frac{e^{ t} }{e^{ t}-1   } - 6 \kappa_{\max} + 11 \kappa_{\min}  > 0   . 
\end{equation*}
What remains is to show that this inequality holds after some amount of time.

To establish this, we consider the evolution of the Gaussian curvature under the normalized Ricci flow, which is given by
\begin{equation*}
    \frac{\partial}{\partial t} \kappa = \Delta \kappa + \kappa^2 - \kappa.
\end{equation*}
Applying the maximum principle, we find that  \[\kappa_{\max{}}(t) < \frac{1}{1-C_0 e^t}, \]
where $C_0 = 1-\frac{1}{\kappa_{\max{}}(0)}$
and that 
\[\kappa_{\min{}}(t) > \frac{1}{1-C_1 e^t}, \]
where $C_1 = 1-\frac{1}{\kappa_{\min{}}(0)}$.

On the other hand, the pinching $\kappa_{\max}-\kappa_{\min}$ satisfies 
\begin{align*}
 \frac{\partial}{\partial t}  (\kappa_{\max}(t)-\kappa_{\min}(t)) & \leq  \kappa_{\max}^2(t) - \kappa_{\max}(t) -\kappa_{\min}^2(t) + \kappa_{\min}(t) \\
 & =  (\kappa_{\max} + \kappa_{\min}-1) (\kappa_{\max}(t)-\kappa_{\min}(t)) \\
 & <  \frac{1}{1-C_0 e^t} (\kappa_{\max}(t)-\kappa_{\min}(t)).
\end{align*}

In other words, for a short period of time, the curvature pinching grows at most exponentially, where the exponential constant is bounded by an upper bound on the sectional curvature.

If the initial curvature is $\alpha$-pinched and the average curvature is $1$, then the maximum principle provides curvature bounds up to time $t=\log(\frac{1}{1-\alpha})$. We can solve the differential equations to find explicit upper bounds on $\kappa_{\max}$ and $\kappa_{\max}(t)-\kappa_{\min}(t)$ as well as lower bounds on $\kappa_{\min}$.

Doing so, we have that
\[\kappa_{\max} < \frac{1}{1-C_0 e^t}, \quad  \kappa_{\min} > \frac{1}{1+C_1 e^t}, \]
and 
\[\kappa_{\max}-\kappa_{\min} <  C_1 \frac{e^t}{1-C_0 e^t}. \]
where $C_0 = 1-\alpha$ and $C_1 = \frac{1}{\alpha}-1$. Putting this together, we have that
\begin{equation*}
    1 -  \frac{e^{ t} }{e^{ t}-1   } - 6 \kappa + 11 \kappa_{\min}  >  1 -    \frac{e^{ t} }{e^{ t}-1   } - 6 C_1 \frac{e^t}{1-C_0 e^t} + 5 \frac{1}{1+C_1 e^t}.
\end{equation*}
Whenever this quantity is positive, the surface is spectral dominant. For instance, when $\alpha=\frac{6}{7}$, at $t=\log(3/2)$, 
\[1 -    \frac{e^{ t} }{e^{ t}-1   } - 6 C_1 \frac{e^t}{1-C_0 e^t} + 5 \frac{1}{1+C_1 e^t} = \frac{1}{11}, \]
which implies any surface which is $\frac{6}{7}$ pinched becomes spectral dominant at time $t=\log\left(\frac{3}{2}\right)\approx .405$.
\end{proof}

Note that this argument does not rule out the possibility that the metrics fail to satisfy the spectral dominance condition at some later time. For instance, for metrics which are initially $\frac{6}{7}$-pinched, this analysis no longer implies spectral dominance after $t \approx .67485$. Since the metric is converging exponentially quickly to the round sphere in $C^4$, there is a time $T$ after which the solution will always satisfy spectral dominance. However, we are currently unable to make quantitative estimates on the time $T$.

\end{document}